\documentclass{amsart}

\usepackage{graphicx}
\usepackage{xcolor}

\newtheorem{theorem}{Theorem}[section]
\newtheorem{lemma}[theorem]{Lemma}
\newtheorem{prop}[theorem]{Proposition}
\newtheorem{cor}[theorem]{Corollary} 

\theoremstyle{definition}
\newtheorem{definition}[theorem]{Definition}
\newtheorem{example}[theorem]{Example}

\theoremstyle{remark}
\newtheorem{remark}[theorem]{Remark}

\numberwithin{equation}{section}

\DeclareMathOperator*{\argmin}{arg\,min}

\newcommand{\h}{\mathcal{H}}
\newcommand{\D}{\mathbb{D}}
\newcommand{\C}{\mathbb{C}}

\newcommand{\ksf}{\mathcal{K}_{Sf}}
\newcommand{\bb}{\overline{\beta}}

\newcommand{\editA}[1]{{\color{black}#1}}
\newcommand{\editB}[1]{{\color{black}#1}}

\begin{document}

\title[General OPAs, Stabilization, and Projections of Unity]{General Optimal Polynomial Approximants, \\Stabilization, and Projections of Unity}

\author{Christopher Felder}
\address{Department of Mathematics and Statistics, Washington University In St. Louis, St. Louis, Mo, 63136}
\email{cfelder@wustl.edu}

\subjclass[2010]{Primary 46E22; Secondary 30J05}
\date{\today}

\keywords{Optimal polynomial approximants, inner functions}

\begin{abstract}
In various Hilbert spaces of analytic functions on the unit disk, we characterize when a function has optimal polynomial approximants given by truncations of a single power series or, equivalently, when the approximants stabilize. We also introduce a generalized notion of optimal approximant and use this to explicitly compute orthogonal projections of 1 onto certain shift invariant subspaces. 
\end{abstract}
\maketitle
\section{Background, Introduction, and Notation}
Throughout this paper $\h$ will be a reproducing kernel Hilbert space of analytic functions on the unit disk $\D$. We will denote the reproducing kernel for $\h$ as $k_\lambda(z) = k(z,\lambda)$ and the normalized reproducing kernel as $\hat{k}_\lambda = k_\lambda/\|k_\lambda\|_\h$. That is, a priori, for $\lambda \in \D$, we have $f(\lambda) = \langle f, k_\lambda \rangle_\h$. Further, we will assume that $\h$ satisfies the following: 
\begin{enumerate}
\item The polynomials $\mathcal{P}$ are dense in $\h$. 
\item The forward shift $S$, mapping $f(z) \mapsto zf(z)$, is a bounded operator on $\h$. 
\end{enumerate}
When $V \subseteq \h$ is a closed subspace, we will use $\Pi_V : \h \to V$ to denote the orthogonal projection from $\h$ onto $V$. For $n \in \mathbb{N}$, we will denote by $\mathcal{P}_n$ the set of complex polynomials of degree less than or equal to $n$. For $f \in \h$, we define $f\mathcal{P}_n := \{ pf : p \in \mathcal{P}_n\}$. Note that $f\mathcal{P}_n$ is always a closed finite-dimensional subspace of $\h$. When $f$ is fixed, we will use $\Pi_n : \h \to f\mathcal{P}_n$ to denote the orthogonal projection onto $f\mathcal{P}_n$. 

\subsection{Cyclicity and Shift Invariant Subspaces}
The results to come are born from the study of shift invariant subspaces and cyclic functions. We say a subspace $V \subseteq \h$ is \textit{shift invariant} if $SV \subseteq V$. We say a function $f \in \h$ is \textit{cyclic} (in $\h$) if 
$$[f] := \overline{\text{span}\{z^nf : n \ge 0\}}^{\ \h}$$ 
is equal to $\h$ itself. Note that $[f]$ is a (possibly trivial) shift invariant subspace and is the smallest closed subspace of $\h$ containing $f$. 
In \cite{brown1984cyclic}, it was pointed out that $f \in \h$ is cyclic if and only if, for any cyclic function $g \in \h$, there exist polynomials $(p_n)_{n \ge 0}$ so that $\| p_nf - g \|_{\h} \to 0$. From this equivalence, and taking $g = 1$ in spaces where $1 = k_0$, the study of \textit{optimal polynomial approximants} has arisen. The optimality referred to here is with respect to the distance between $f \mathcal{P}_n$ and 1, i.e., 
\[
\min_{p \in \mathcal{P}_n} \| pf - 1 \|_\h.
\] 
The \editA{element of $f\mathcal{P}_n$} minimizing this distance will be denoted $p_n^*f$ (details to come in Section \ref{genopa}).

\editA{Approximation problems of this kind were first studied under the engineering lens of filter design in the 1970's and 80's, referred to as \textit{least squares inverses} (see, e.g. \cite{1162358, MR576280, 1163842}). It seems this body of work was not known to mathematicians prior to the discussion in \cite{beneteau2019boundary}.}

A \editA{modern} jumping off point for optimal approximants could be considered the work in \cite{fricain2014cyclicity}; the authors study the optimal approximants of the function $1-z$ in order to characterize the cyclicity of holomorphic functions on the closed unit disk. 
In \cite{beneteau2019boundary}, the authors compute Taylor coefficients of $1 - p_n^*f$ in weighted Hardy spaces (discussed below) when $f$ is a polynomial, proving results about the convergence of $( 1 - p_n^*f )$. 

\editA{In \cite{beneteau2016zeros}, the authors study a larger class of reproducing kernel Hilbert spaces and give results on accumulation points, along with lower bounds on the moduli of zeros of optimal approximants. 
Then in \cite{beneteau2016orthogonal}, the authors dive into orthogonal polynomials and reproducing kernels in order to get lower bounds on the moduli of zeros of optimal approximants in Dirichlet-type spaces. }

Following these themes, we would like to develop some theory for different choices of $g$ (cyclic or not) in considering $\|pf - g \|_\h$, and then explore the relationship between optimal approximants and generalized \textit{inner} functions (\editB{this relationship first studied in \cite{beneteau2017remarks}}). This will then yield some observations which allow us to explicitly compute $\Pi_{[f]}(1)$ when $f$ is a polynomial. 

\editB{
In particular:
\begin{itemize}
\item Section \ref{genopa} develops the framework necessary for handling general optimal approximants. 
\item Section \ref{stable} deals with \textit{stabilization} of optimal approximants to $\hat{k}_0/f$, with Theorem \ref{main} characterizing when $p_n^*f = p_M^*f$ for all $n$ great than some fixed $M\ge 0$. 
\item Section \ref{g-opas} discusses stabilization of general optimal approximants, with Theorem \ref{gen-main} giving a version of Theorem \ref{main} for general approximants. 
\item Section \ref{projofunity} develops the theory of \textit{reproducible points}, and then returns to certain spaces where $\hat{k}_0 = 1$, with Theorem \ref{projun} providing an explicit description of the projection of 1 onto the shift invariant subspace generated by a polynomial.
\end{itemize}

Many of the themes of this paper follow from those in \cite{beneteau2017remarks}. The authors there show that inner functions correspond to constant optimal approximants and investigate certain inner functions that arise as linear combinations of reproducing kernels. 
}

We conlcude this section by mentioning some spaces where assumptions (1) and (2) from above hold.

%
%

\subsection{Weighted Hardy Spaces}\label{whs}
A well-studied family of spaces satisfying these properties are some weighted Hardy spaces. Letting $w := \{w_k\}_{k \ge 0}$ be a sequence of positive real numbers with $\lim_{k \to \infty}w_{k+1}/w_k =  1$ and $w_0 =1$, define $H^2_w$ as the space of all functions $f(z)$ with Maclaurin series 
\[
f(z) =  \sum_{k=0}^{\infty} a_k z^k, \ \ \ |z| < 1
\]
for which 
\[
\| f \|^2_w := \sum_{k=0}^{\infty} w_k |a_k|^2 < \infty.
\]
We point out that $H^2_w$ is a Hilbert space; if $f$ and $g$ are elements of $H^2_w$ with Maclaurin coefficients $\{a_k\}_{k\ge0}$ and $\{b_k\}_{k\ge0}$ respectively, their inner product is given by 
\[
\langle f, g \rangle_w = \sum_{k = 0}^{\infty}w_k a_k \overline{b_k}.
\]
The limit condition on the sequence $w$ ensures that functions analytic in a disk larger than $\D$ belong to $H^2_w$, and that all functions in these spaces are analytic in $\D$. Taking $\alpha \in \mathbb{R}$ and $w = \{\left(k+1\right)^\alpha\}_{k \ge 0}$ gives the Dirichlet-type spaces $\mathcal{D}_\alpha$. When $\alpha = 0$ we recover the classical Hardy space $H^2$, $\alpha = -1$ gives the Bergman space $A^2$, and $\alpha = 1$ gives the Dirichlet space $\mathcal{D}$. Much of the existing literature on optimal polynomial approximants has focused on these spaces. However, in this paper, the results to be proved will extend to some other spaces that do not have some of the useful properties present in the $\mathcal{D}_\alpha$ spaces. Below we give two examples of such spaces. 

\subsection{Szeg\H{o}'s Theorem and $\frac1m H^2$}\label{oneoverm}
A classical theorem of Szeg\H{o} says that for $v \in L^1(\mathbb{T})$ positive, the closure of the analytic polynomials in $L^2(v)$ coincides with all of $L^2(v)$ if and only if $\int_{\mathbb{T}} \log{v} = -\infty$ (e.g., see \cite{conway1991theory}). In the case that $\int_{\mathbb{T}} \log{v} > -\infty$, there exists an outer (i.e.\editA{,} $H^2$-cyclic) function $m$ such that $v = |m|^2$. Further, $P^2(v) := \overline{\text{span}\{z^k : k \ge 0\}}^{\ L^2(v)}$ is isomorphic to $\frac1m H^2 : = \{ f/m : f \in H^2 \}$ \editB{(which we endow with the $H^2$ norm)}. It follows that multiplication by $1/m$ is an isometry and for all $f\in P^2(v)$, we have
$\| f \|_{P^2(v)} = \| f/m \|_{H^2}$.
A \editA{distinctive characteristic of} these spaces is that the monomials are not pairwise orthogonal. 

\subsection{de Branges-Rovnyak Spaces}
Denote by $H^\infty$ the set of bounded analytic functions on $\D$. If $b$ is a function in the unit ball of $H^\infty$ (i.e.\editA{,} $\sup_{z\in \D}|b(z)| \le 1$), then there exists a reproducing kernel Hilbert space on $\D$, denoted $\h(b)$ so that the reproducing kernel for this space is given by
\[
k_\lambda(z) = \frac{1 - \overline{b(\lambda)}b(z)}{1 - \overline{\lambda}z}.
\]
These spaces are called \textit{de Branges-Rovnyak spaces} (see \cite{timotin2015short} for an introduction). The structure of these spaces \editA{varies} with \editA{the} choice of $b$; we would like to keep in mind the spaces for which  the reproducing kernel at zero is not equal to 1 (i.e.\editA{,} when $b(0) \neq 0$). We will generalize some ideas from the existing body of work, for example in the Dirichlet-type spaces, where the function 1 is the reproducing kernel at zero. We will not dig into the study of de \editB{Branges}-Rovnyak spaces here, but the authors in \cite{fricain2014dbr} have characterized cyclicity when $b$ is non-extreme.

\section{General Optimal Approximants}\label{genopa}
We make the distinction of \textit{general} optimal polynomial approximant to generalize the case when $g=1$ in studying $\| pf - g \|_\h$. Any further use of $g$ will be in this context. We will now lay the framework for studying such approximants.

\begin{definition}[Optimal Polynomial Approximant]
\editB{Let $f,g \in \h$ and $n \in \mathbb{N}$.} Define the $n$th \textit{optimal polynomial approximant} to $g/f$ as 
\[
p_n^* := \argmin_{p \in \mathcal{P}_n} \| pf - g \|_\h.
\]
\end{definition}
\editA{Here, $\argmin$ is the argument of the minimum, i.e., 
\[
p_n^* = \{ p \in \mathcal{P}_n : \| pf - g \|_\h \le \| qf - g \|_\h \ \text{for all} \ q\in \mathcal{P}_n\}.
\]}Given the Hilbert space structure, the above minimization is immediate-- simply project $g$ onto the closed subspace $f \mathcal{P}_n$, i.e.\editA{,}
\[
p_n^* \editB{f} = \Pi_{f\mathcal{P}_n}(g).
\]
\editB{Hence, the solution to the minimization problem uniquely exists so long as $f$ is not identically zero, and is non-zero so long as $g$ is not orthogonal to $f\mathcal{P}_n$. In turn, we will be mostly concerned with the cases where $f \not\equiv 0$ and $g$ is not orthogonal to $f\mathcal{P}_n$ for some $n\ge0$. We note that when $g$ is chosen to be the reproducing kernel at the origin, we have that $k_0$ is orthogonal to $f\mathcal{P}_n$ (for any $n\ge 0$, and in the limit) if and only if $f$ and $k_0$ are orthogonal, i.e., $f(0) = 0$.} \editA{Intuitively}, if $\lim_{n \to \infty} p_n^*$ looks like $g/f$, then the above norm goes to zero and does so \textit{optimally}. In this sense, we are trying to approximate $g/f$ with polynomials. 

In \cite{fricain2014cyclicity} (Theorem 2.1)\editA{,} an algorithm for finding optimal polynomial \editA{approximants} is given for $g = 1$ in spaces where $k_0$, the reproducing kernel at zero, is equal to 1. We generalize the ideas from this algorithm below.

\begin{definition}[Optimal System]
For $f, g \in \h$, define the $n$th \textit{optimal matrix} of $f$ in $\h$ as 
\[
G_n := \left( \langle z^if, z^jf \rangle_\h \right)_{0 \leq i,j \leq n}
\]
and the $n$th \textit{optimal system}  of $g/f$ as 
\[
G_n  \editA{\vec{x}} = \left( \langle g, f \rangle, \langle g, zf \rangle, \ldots, \langle g, z^nf \rangle\right)^T.
\]
\end{definition}

The following proposition will shed light on these definitions.

\begin{prop}
\editA{
Let $f,g \in \h$. The vector $\vec{a}_n = (a_0, a_1,  \ldots, a_n)^T$ solving the optimal system
\[
G_n \editA{\vec{x}}=  \left( \langle g, f \rangle, \langle g, zf \rangle, \ldots, \langle g, z^nf \rangle\right)^T
\]
gives the coefficients of the $n$th optimal approximant to $g/f$.
That is, the $n$th optimal approximant to $g/f$ is $p_n^*(z) = a_0 + a_1z + \cdots + a_n z^n$.
}
\end{prop}
\begin{proof}
The optimality of $p_n^*$ means for all $q \in \mathcal{P}_n$
\[
\| p_n^*f - g \|_\h^2 \leq \| qf - g \|_\h^2.
\]
This occurs if and only if $p_n^*f - g \perp qf$. Equivalently, for $j = 0, \ldots, n$, we must have
\[
\langle p_n^*f - g, z^j f \rangle_\h = 0.
\]
Moving $\langle g, z^j f \rangle_\h$ to the right hand side of the above equation and putting $p_n^*(z) = \sum_{j=0}^n a_j z^j$ gives the proposed system. 
\end{proof}

Our next proposition is well-known and will be important for our work; for posterity, we provide a proof. 
\begin{prop}
For $f \in \h$, the orthogonal projections $\Pi_n : \h \to f\mathcal{P}_n$ converge to the orthogonal projection $\Pi_{[f]} : \h \to [f]$ in the strong operator topology. 
\editB{Further, if $f, g\in \h$ with $f\not \equiv 0$, and $(p_n^*)_{n \ge 0}$ the optimal approximants to $g/f$}, then $\varphi := \Pi_{[f]}(g)$ is the unique function such that
\[
\| p_n^*f - \varphi \|_\h \to 0. 
\]
\end{prop}
\begin{proof}
Let $u \in \h$ and put $u = \Pi_{[f]}(u) + v$. Then $v$ is orthogonal to $[f]$, and hence orthogonal to $f\mathcal{P}_n$, so $\Pi_n (v) = 0$ for all $n \ge 0$. Since $\cup_n f\mathcal{P}_n$ is dense in $[f]$, given $\epsilon > 0$, there exists $N$ such that $\text{dist}(\Pi_{[f]}(u), f\mathcal{P}_N) < \epsilon$. Then, for all $ n \ge N$, we have 
\begin{align*}
\| \Pi_{[f]}(u) - \Pi_n(u) \|_\h
&= \text{dist}(\Pi_{[f]}(u), f\mathcal{P}_n) \\
& \le \text{dist}(\Pi_{[f]}(u), f\mathcal{P}_N) \\
&< \epsilon.
\end{align*}
Since $u$ was arbitrary, we have that $\Pi_n \to \Pi_{[f]}$ strongly. 

Further, take $u=g$ to get $\| \Pi_n(g) - \Pi_{[f]}(g) \|_\h = \| p_n^*f - \varphi \|_\h \to 0$. 
\end{proof}

\editA{Again, note that if $g$ is cyclic, then $f$ is cyclic if and only if $p_n^*f \to g$, where $(p_n^*)_{n\ge0}$ are the optimal approximants to $g/f$}. We will now make some observations and motivate a few questions surrounding the behavior of optimal approximants.

\section{Truncations of Power Series and Stabilization of Optimal Approximants}\label{stable}

Let $h$ be analytic on some domain containing the origin. We will denote the $n$th Taylor polynomial of $h$ as
\[
T_n\left(h\right) := \sum_{k=0}^{n} \frac{ h^{\left(k\right)}\left(0\right)}{k!} z^k.
\]
For $f\in \h$, a first natural guess might be that the optimal approximants to $g/f$ are $T_n(g/f)$. \editA{However, it turns out that Taylor polynomials are a poor guess}. For example, in the Dirichlet space $\mathcal{D}$, the cyclic function $1-z$ was studied in \cite{beneteau2015cyclicity}, and there it was pointed out that
\begin{align*}
\| T_n(1/f)f - 1 \|_{\mathcal{D}} &= \| \left(1 + z + \ldots + z^n\right)(1-z) - 1 \|_{\mathcal{D}}\\
&= \| z^{n+1} \|_{\mathcal{D}}\\
&= n+1,
\end{align*}
which is unbounded as $n \to \infty$. In this case, $T_n(1/f)$ is neither optimal nor provides a sequence that proves $f$ to be cyclic (even though $T_n(1/f)f \to 1$ pointwise in $\D$). 
Instead of using Taylor polynomials, we ask a couple of more general questions:
\begin{enumerate}
\item[(Q1)] Given \editA{$g \in \h$ and} a power series $\varphi(z) = \sum_{k=0}^{\infty} a_k z^k$, can we characterize $f \in \h$ such that the $n$th optimal polynomial approximants to $g/f$ are given by $T_n(\varphi)$ for all $n$ greater than some $M>0$?
\item[(Q2)] \editA{Given $g\in \h$ and supposing $p$ is a polynomial, can we characterize $f$ such that $\Pi_{[f]} (g) = p f$? }
\end{enumerate}

We will proceed by first answering these questions when $g = \hat{k}_0$.

\subsection{The Reproducing Kernel at Zero and Inner Functions}
As mentioned previously, much of the existing literature on optimal approximants has been centered around approximating $1/f$ in spaces where 1 is the reproducing kernel at zero. In the \editB{present} section, we will make a few observations and generalize these results. 

\begin{definition}[\editA{$\h$-inner function}]
Say that $f \in \h$ is \textit{$\h$-inner} if
\[
\langle f, z^jf  \rangle_\h = \delta_{j0}.
\]
\end{definition}
This definition was first given by Aleman, Richter, and Sundberg in \cite{aleman1996beurling} for $\h = A^2$, and coincides with the classical definition of inner in $H^2$\editA{; a function is $H^2$-inner if $|f| = 1$ always everywhere on the unit circle}. Classical inner functions play a crucial role in understanding operator and function theoretic properties of $H^2$. Let us gather some facts about the relationship between \editA{$\h$-inner} functions and optimal polynomial approximants. Again, we point to \cite{beneteau2017remarks} for further discussion on this topic, where it was first studied. 

\begin{prop}\label{cyclic-inner}
\editB{If there is a function in} $\h$ that is both cyclic and \editA{$\h$-inner}, \editB{then, up to a unimodular constant, this function is unique, and is the normalized reproducing kernel at zero}.
\end{prop}
\begin{proof}
Let $\theta \in \h$ be cyclic and \editA{$\h$-inner}. Then for all $h \in \h$, there exist polynomials $p_n$ such that $p_n \theta \to h$ and as $\theta$ is \editA{$\h$-inner}, $\langle  p_n \theta, \theta \rangle_\h = p_n(0)$. Taking limits, and noting $\theta(0) \neq 0$ by cyclicity, we have $\langle h, \theta \rangle_\h = h(0)/\theta(0)$. 
This implies that $\overline{\theta(0)}\theta$ is the reproducing kernel at zero. Thus, by the Riesz representation theorem, this function is well-defined for any choice of $\theta$ and must be $k_0$. Normalizing \editB{$\overline{\theta(0)}\theta$} then concludes the proof. 
\end{proof}

\editB{In general, the kernel at the origin is always $\h$-inner, but it is not known if it must also be cyclic (hence, the existence hypothesis in the above proposition).}
Note that in the Dirichlet-type spaces, the functions $\theta$ above are just unimodular constants, \editB{and $k_0=1$ is clearly cyclic}. However, as noted previously, in DeBrange-Rovnyak spaces $\h(b)$, unless $b(0) = 0$, the reproducing kernel at zero is non-constant and is given by $\overline{\theta(0)}\theta= 1 - \overline{b(0)}b$. \editB{Even in this case, it is not known if the kernel at zero must always be cyclic.}

\editB{We mention again that the optimal approximants to $\hat{k}_0/f$ are non-zero if and only if $f(0) \neq 0$.}

\begin{lemma}\label{inner}
Let $f \in \h$ with $f(0) \neq 0$. Let $\varphi$ be the orthogonal projection of $k_0$ onto $[f]$. Then $\varphi/\sqrt{\varphi(0)}$ is $\h$-inner. 
\end{lemma}

\begin{proof}
Notice that $k_0 - \varphi \perp [f]$ and $[f]$ is shift invariant, so for all $j \ge1$ we have 
\[
0 = \langle z^j \varphi, k_0 - \varphi \rangle_\h = - \langle z^j \varphi, \varphi \rangle_\h.
\]
Further, $\langle \varphi, \varphi \rangle_\h = \langle k_0, \varphi \rangle_\h = \varphi(0)$ which gives $\|\varphi\|_\h = \sqrt{\varphi(0)}$. Thus
\[
\left\langle \frac{\varphi}{\sqrt{\varphi(0)}}, z^j \frac{\varphi}{\sqrt{\varphi(0)}} \right\rangle_\h = \delta_{j0}
\]
so $\varphi/\sqrt{\varphi(0)}$ is \editA{$\h$-inner}. 
\end{proof}

\begin{lemma}\label{finite}
Let $f \in \h$ with $f(0) \neq 0$ and let $(p_n^*)$ be the optimal approximants to $\hat{k}_0/f$. Let $\varphi(z) = \sum_{k=0}^{\infty}a_kz^k$ and suppose that $p_n^* = T_n(\varphi)$ for all $n \ge M$. Then $p_n^* = p_M^*$ for all $n \ge M$. That is, $\varphi = p_M^*$.
\end{lemma}
\begin{proof}
By hypothesis, for all $n \ge M$, $\varphi(0) = (p_n^* f)(0) = (p_M^*f)(0)$. Now notice, for all $n \ge M$,
\begin{align*}
\| p_n^* f - p_M^*f \|_\h^2 &= \| p_n^* f \|_\h^2 - 2\text{Re}\{\langle p_n^*f, p_M^*f \rangle_\h \} + \| p_m^* f \|_\h^2\\
&= (p_n^* f)(0) - 2(p_M^*f)(0) + (p_M^*f)(0)\\
&=0
\end{align*}
Hence, $p_n^*f = p_M^*f$ for all $n \ge M$, and as $f$ is not identically zero, $p_n^* = p_M^*$ for all $n \ge M$.\\
\end{proof}

\begin{remark}
It should be pointed out that Lemma \ref{finite} says that there are no functions $f$ for which the optimal approximants to $\hat{k}_0/f$ come from truncations of a single power series with finitely many zero coefficients. This lemma can also be seen as a consequence of the simple exercise showing that $\operatorname{dist}^2(\hat{k}_0, f\mathcal{P}_n) = 1 - (p_n^*f)(0)$.
This also tells us that for $g = \hat{k}_0$, (Q1) and (Q2) are equivalent. The following definition is now natural.
\end{remark}

\begin{definition}[Stabilizing approximants]
Let $f,g \in \h$ with \editB{$g$ not orthogonal to $[f]$} and let $(p_n^*)_{n \ge 0}$ be the optimal approximants to $g/f$. Say that the optimal approximants \textit{stabilize} at $p_M^*$ if $M$ is the smallest non-negative integer such that $p_n^* = p_M^*$ for all $n \ge M$. 
\end{definition}

\begin{lemma}\label{constant}
Let $f \in \h$ with $f(0) \neq 0$ and let $(p_n^*)_{n \ge 0}$ be the optimal approximants to $\hat{k}_0/f$. Then $f$ is \editA{$\h$-inner} \editB{(up to a constant multiple)} if and only if, for all $n \ge 0$,
\[
\editB{p_n^* = \frac{\overline{f(0)}}{\|k_0\|\|f\|^2}}.
\]
\end{lemma}
\begin{proof}
For the forward direction, suppose $f$ is \editB{a constant multiple of an} \editA{$\h$-inner function}. For any $n \ge 0$, consider the optimal system for $\hat{k}_0/f$:
\[
G_n \vec{x} = \left( \langle \hat{k}_0, f \rangle, 0, \ldots, 0\right)^T =  \left( \|k_0\|^{-1}\overline{f(0)}, 0, \ldots, 0\right)^T.
\]
As $\langle f, z^kf \rangle = 0$ for all $k \ge 1$, the entries in the first row and column of $G_n$, except the (0,0) entry, are all zero. It follows that the inverse of $G_n$ must also satisfy this property. Now, considering $G_n^{-1}\left( \|k_0\|^{-1}\overline{f(0)}, 0, \ldots, 0\right)^T$ to recover the coefficients of $p_n^*$, we see that $p_n^*$ is the constant $\frac{\overline{f(0)}}{\|k_0\|\|f\|^2}$ for any $n\ge0$. 

Now suppose $p_n^*(z) = \frac{\overline{f(0)}}{\|k_0\|\|f\|^2}$ for all $n \ge 0$. Considering the optimal system 
\[
G_1 \left(\frac{\overline{f(0)}}{\|k_0\|\|f\|^2}, 0\right)^T = \left(\frac{\overline{f(0)}}{\|k_0\|}, 0\right)^T
\]
quickly yields that $\langle f, zf \rangle_\h = 0$. As the coefficients of $p_n^*$ are stable, a simple induction argument then shows that  $\langle f, z^k f \rangle_\h = 0$ for all $k \ge 1$. Thus, $f$ is a constant multiple of an $\h$-inner function.
\end{proof}

The forward implication of this lemma was given in \cite{beneteau2017remarks} for spaces where $\hat{k}_0 = 1$. 
We now give a characterization of stabilizing approximants, which answers (Q2) when $g = \hat{k}_0$. 

We are now ready to prove the main theorem of this section, which gives a characterization of functions with stabilizing approximants. 
\begin{theorem}\label{main}
Let $f \in \h$ with $f(0) \neq 0$ and let $(p_n^*)$ be the optimal polynomial approximants to $\hat{k}_0/f$. The following are equivalent, and the smallest $M$ for which each of the statements hold is the same: 
\begin{enumerate}
\item There exists a function $\varphi(z) = \sum_{k \ge 0} a_k z^k$ such that $p_n^* = T_n(\varphi)$ for all $n \ge M$.
\item The optimal approximants to $\hat{k}_0/f$ stabilize at $p_M^*$.
\item $p_M^*f$ is the orthogonal projection of $\hat{k}_0$ onto $[f]$.
\item $f = c u/p_M^*$, where $c = \sqrt{(p_M^*f)(0)}$ and $u$ is $\h$-inner.
\end{enumerate}
\end{theorem}

\begin{proof}
The equivalence of (1) and (2) is given by Lemma \ref{finite} and taking $p_M^* = \varphi$ for the backward implication. The equivalence of (2) and (3) follows by definition. The fact that (3) implies (4) is given by Lemma \ref{inner}. The unique minimality of $M$ until now follows by definition and trivial arguments.

Now let us assume (4), putting $p_M^*(z) = \sum_{k=0}^M a_k z^k$ and assuming that $M$ is minimal. Then, 
\begin{align*}
0 &=  \left\langle z \ \frac{p_M^*f}{\sqrt{(p_M^*f)(0)}}, \frac{p_M^*f}{\sqrt{(p_M^*f)(0)}}\right\rangle_\h \\
&= \langle z p_M^*f, p_M^*f \rangle_\h \\
&= \sum_{k=0}^M a_k \langle z^{k+1} f, p_M^*f \rangle_\h \\
&= a_M \langle z^{M+1} f, p_M^*f \rangle _\h
\end{align*}
where the last equality holds by optimality of $p_M^*$. 
By the minimality of $M$, $a_M \neq 0$ so we must have $\langle z^{M+1} f, p_M^*f \rangle_\h = 0$. A simple induction argument shows that $\langle z^{M+k} f, p_M^*f \rangle_\h = 0$ for all $k \ge 1$. It follows that 
\[
\langle qf, p_M^*f \rangle_\h = q(0)f(0)
\]
for all $q \in \mathcal{P}$. In other words, $p_M^*f$ is the orthogonal projection of $\hat{k}_0$ onto $[f]$, i.e.\editA{,} (3) \editA{holds}. 
\end{proof}

As previously mentioned, much effort has gone into understanding the location of zeros of optimal approximants. We end this section by showing that if the kernel at the origin is cyclic, then stable approximants must have zeros which are outside of the open unit disk.

\begin{cor}\label{inv}
Let $f \in \h$ be cyclic \editB{and suppose that $k_0$ is cyclic in $\h$. If the optimal polynomial approximants to $\hat{k}_0/f$ stabilize at $p_M^*$, then $f = \hat{k}_0/p_M^*$, and $p_M^*$ has no zeros inside $\D$.}
\end{cor}
\begin{proof}
Since $f$ is cyclic, $f(0) \neq 0$. By optimality, we have 
\[
\langle p_M^* f, qf \rangle_\h = \langle \hat{k}_0, qf \rangle_\h
\]
for all $q \in \mathcal{P}$. As $f$ is cyclic, $\{qf : q \in \mathcal{P} \}$ is dense in $\h$. It follows immediately that $p_M^*f  = \hat{k}_0$. Lastly, as $\hat{k}_0$ is assumed cyclic, and therefore zero-free in $\D$, and $f$ is analytic in $\D$, $p_M^*$ must not have any zeros in $\D$. 
\end{proof}

\begin{remark}\label{zeros}
For $h\in \h$, let us denote the zero set of $h$ as 
\[
Z(h) := \{ \beta \in \operatorname{Dom}(h) : h(\beta) = 0 \}.
\]
It was shown in \cite{beneteau2016orthogonal} that in the Dirichlet-type spaces $D_\alpha$, $Z(p_n^*) \cap \overline{\mathbb{D}} = \emptyset$ when $\alpha \ge 0$ and $Z(p_n^*) \cap \overline{D}(0, 2^{-\alpha/2}) = \emptyset$ when $\alpha < 0$. The above corollary improves this result for $\alpha < 0$ when $f$ is cyclic and has stabilizing approximants. However, it should be noted that, a priori, $p_m^*$ may have zeros on the unit circle. 
\end{remark}



\section{General Approximants}\label{g-opas}
We now return to the case of approximating some arbitrary $g/f$ with $g,f \in \h$. Recalling the $\frac1m H^2$ spaces from Section \ref{oneoverm}, which serve as one motivation for studying general approximants, 
we have the following proposition.
\begin{prop}
Let $f \in \frac1m H^2 \setminus \{0\}$. Put $f = h/m$ with $h\in H^2$. Then the optimal polynomial approximants to $1/f$ in $\frac1m H^2$ correspond to the optimal polynomial approximants to $m/h$ in $H^2$. 
\end{prop}
\begin{proof}
\editB{Recall that multiplication by $m$ is an isometry from $\frac1m H^2$ to $H^2$,} and notice that for any polynomial $p$ we have
\[
\| pf - 1 \|_{\frac1m H^2} = \| ph - m \|_{H^2}.
\]
Minimizing each side of the equality above we see that 
\[
\left(\Pi_{f\mathcal{P}_n}(1)\right)/f = \left(\Pi_{h\mathcal{P}_n}(m)\right)/h,
\]
where the projections on the left and right hand sides above are taken in $\frac1m H^2$ and $H^2$, respectively.
Lastly, as $f \not\equiv 0$, these projections are unique and represent the optimal approximants. 
\end{proof}
We can now reframe questions about cyclicity in $\frac1m H^2$ as questions in $H^2$ via general optimal approximants. This is advantageous because $H^2$ has nicer structural properties than $\frac1m H^2$ (e.g., the monomials are orthogonal in $H^2$ but not in $\frac1m H^2$). 

Let us now give some results pertaining to \editA{$\h$-inner} functions and general optimal approximants. In general, (Q1) and (Q2) are not equivalent. For example, if $f(z) = 1$ and $g(z) = \sum_{k\ge 0} b_kz^k$, then the optimal approximants to $g/f$ are just $T_n(g)$, since $\Pi_{f\mathcal{P}_n}(g) = \Pi_{\mathcal{P}_n}(g) = T_n(g)$.

\subsection{General Stabilization}
The aim of this section is to provide a stabilization theorem for a certain class of general approximants. 
We will be able to do so with the help of the following proposition, which deals with the orthogonal complement of the subspace generated by $zf$. When $f \in H^2$ is inner, these spaces are examples of \textit{model spaces} (see, e.g., \cite{timotin2015short} for an introduction). 
\begin{prop}\label{modelshift}
Let $f \in \h$ and define
\[
\ksf : = \h \ominus [Sf].
\]
For any $h \in \h$, we have $h \in \ksf$ if and only if $\Pi_{[f]}(h) \in \ksf$. Further, $\hat{k}_0$ and $\Pi_{[f]}(\hat{k}_0)$ are always elements of $\ksf$. 
\end{prop}

\begin{proof}
Note that $\ksf$ can also be expressed as 
\[
\ksf = \{ h \in \h : \langle h, z^kf \rangle_\h = 0 \ \text{for all} \  k \ge 1\}.
\] 
Simply observe that $\langle h, z^kf \rangle_\h = \langle h, \Pi_{[f]}(z^kf) \rangle_\h  = \langle \Pi_{[f]}(h), z^kf \rangle_\h$ and that \\
$\langle k_0, z^kf \rangle_\h = 0$ for all $k \ge 1$. 
\end{proof}

We can now prove the main result of this section.

\begin{theorem}\label{gen-main}
Let $f, g \in \h$ with \editB{$g$ not orthogonal to $[f]$}. Let $(q_n^*)$ be the optimal approximants to $g/f$. The following are equivalent, and the smallest $M$ for which each of the statements hold is the same: 
\begin{enumerate}
\item $g \in \ksf$ and $\Pi_{[f]}(g) = q_M^*f$.
\item $q_M^*f \in \ksf$.
\item $q_M^*f/ \|q_M^*f\|_\h$ is \editA{$\h$-inner} and $\langle q_M^*f , z^k f \rangle_\h = 0$ for $k= 1, \ldots, M$.
\end{enumerate}
\end{theorem}
\begin{proof}
To see (1) implies (2), note that if $\Pi_{[f]}(g) = q_M^*f$ then $\langle q_M^*f, z^k f \rangle_\h = \langle g, z^kf \rangle_\h$. So if $g \in \ksf$, then $\langle q_M^*f, z^k f \rangle_\h = 0$ for all $k \ge 1$.

For (2) implies (3), the fact that $\langle q_M^*f , z^k f \rangle_\h = 0$ for $k= 1, \ldots, M$ follows by definition of $q_M^*f \in \ksf$. To see $q_M^*f/ \|q_M^*f\|_\h$ is \editA{$\h$-inner}, put $q_M^*(z) = \sum_{j=0}^M b_kz^j$ and observe, for all $k \ge1$,
\[
\langle q_M^*f, z^kq_M^*f \rangle_\h = \sum_{j=0}^M \overline{b_j} \langle q_M^*f, z^{j+k}f \rangle_\h = 0
\]
where the second equality holds because $q_M^*f \in \ksf$.
Thus, $q_M^*f/ \|q_M^*f\|_\h$ is \editA{$\h$-inner}. Further, the unique minimality of $M$ in the above statements is immediate. 

For (3) implies (1), we use the same idea as the last part of Theorem \ref{main}.
Put $q_M^*(z) = \sum_{j=0}^M b_k z^j$ and assume $M$ is minimal. Since $q_M^*f/\| q_M^*f\|_\h$ is \editA{$\h$-inner}, we have 
\begin{align*}
0 &= \langle z q_M^*f, q_M^*f \rangle_\h \\
&= \sum_{j=0}^M b_k \langle z^{j+1} f, q_M^*f \rangle_\h \\
&= b_M \langle z^{M+1} f, q_M^*f \rangle_\h
\end{align*}
where the last equality holds by the assumption that $q_m^*f$ is orthogonal to $z^kf$ for $k = 1, \ldots, M$. 
By the minimality of $M$, $b_M \neq 0$ so we must have $\langle z^{M+1} f, q_M^*f \rangle_\h = 0$. A simple induction argument shows that $\langle z^{M+k} f, q_M^*f \rangle_\h = 0$ for all $k \ge 1$. Thus, $q_M^*f \in \ksf$. Further, if $q_M^*f \in \ksf$, then $g - q_M^*f = g - \Pi_M(g)$ is orthogonal to $[f]$. It follows that $\Pi_{[f]}(g) = q_M^*f $, thus the approximants to $g/f$ stabilize at $q_M^*$.
Lastly, $g \in \ksf$ since now $\langle q_M^*f, z^kf \rangle_\h = \langle g, z^kf \rangle_\h = 0$ for all $k\ge 1$.
\end{proof}

\editB{When $k_0$ is cyclic in $\h$,} we can also characterize cyclicity in terms of $\ksf$. 
\begin{prop}
Let $f \in \h$ \editB{and suppose that $k_0$ is cyclic in $\h$}. Then $f$ is cyclic if and only if $\ksf = \operatorname{span}\{k_0\}$.
\end{prop}
\begin{proof}
Suppose $f$ is cyclic. Then for any $h \in \h$, we can find polynomials $(p_n)$ so that $p_nf \to h$. Letting $g \in \ksf$ we have 
\begin{align*}
\langle g, h \rangle_\h &= \lim_{n \to \infty} \ \langle g, p_nf \rangle_\h \\
&=\lim_{n \to \infty} \overline{(p_nf)(0)} \ \langle g , 1 \rangle_\h\\
&= \overline{h(0)} \  \langle g , 1 \rangle_\h.
\end{align*}
Thus, $g$ reproduces, up to a constant, the value of $h$ at zero so $g \in \operatorname{span}\{k_0\}$.

Conversely, let $\ksf = \text{span}\{k_0\}$. Since $\Pi_{[f]}(k_0) \in \ksf$, there exists some constant $\lambda$ so that $\Pi_{[f]}(k_0) = \lambda k_0$. This means that the cyclic function $k_0 \in [f]$ so $f$ is cyclic. 
\end{proof}
One may compare this with the well-known ``codimension one" property of shift invariant subspaces (e.g., see \cite{richter1988bounded}). 

\section{Projections of Unity}\label{projofunity}

In light of Proposition \ref{modelshift}, we will compute $\Pi_{[f]}(1)$ (i.e., a projection of unity) when $f \in \mathcal{P}\subset H^2_w$. Note that in our definition of $H^2_w$ from Section \ref{whs}, we have $\hat{k}_0 = 1$. 

\editB{As we will see, these projections are linear combinations of reproducing kernels. 
This idea goes back to a construction of Shapiro and Shields in the Bergman space \cite{MR145082} involving certain Gram determinants, later modified by the authors in \cite{beneteau2017remarks} to produce examples of $H^2_w$-inner functions. The inner functions constructed there are associated to a finite set of distinct points in the open unit disk. We would like to generalize this theory by considering finite sets of points in the plane with any multiplicity. 
This motivates the following definition.} 

\begin{definition}[Reproducible point]
Let $\beta \in \C$ and $m \in \mathbb{Z}^+ \cup \{\infty \}$. Say that $\beta$ is reproducible of order $m$ in $H^2_w$ if point evaluation at $\beta$ of the $m$-th derivative of functions in $H^2_w$ is bounded. If no such $m$ exists, say that $\beta$ is not reproducible. Denote the collection of reproducible points of order $m$ for $H^2_w$ as $\Omega_m(H^2_w)$.
\end{definition}

Notice that $\Omega_0(H^2_w)$ is just the set of points for which point evaluation is bounded in $H^2_w$. Since $H^2_w$ is a reproducing kernel Hilbert space on $\D$, we always have $\D \subseteq \Omega_0(H^2_w)$. But $\Omega_0(H^2_w)$ could be a strictly larger set.
For example, a routine exercise shows that when $\alpha > 1$, $\Omega_0 (\mathcal{D}_\alpha) = \overline{\D}$ and $\Omega_m(\mathcal{D}_\alpha) \subseteq \overline\D$ for all $m \ge 1$ (the proper inclusion depending on $m$ and $\alpha$). If $|\beta| > 1$, then $\beta$ is not reproducible in $\mathcal{D}_\alpha$. 
In $H^2$, $\Omega_m(H^2) = \D$ for all $m$.

We need one more observation and lemma before stating our last theorem. Let $s_\beta(z) = \frac{1}{1-\bb z}$ denote the Szeg\H{o} kernel, which is the reproducing kernel in $H^2$. Let $s_\beta^{(n)}$ denote the $n$-th derivative of $s_\beta$ and let $s_\beta^n$ denote the reproducing kernel for $n$-th derivatives in $H^2$, i.e.\editA{,} $\langle f, s_\beta^n \rangle_{H^2} = f^{(n)}(\beta)$ for all $f \in H^2$. Such an element exists because $f$ is assumed to be analytic, and is unique by the Riesz representation theorem. A simple exercise shows that 
\[
s_\beta^{(n)}(z) = \sum_{j\ge 0} (j+1)(j+2)\ldots(j+n)\bb^{j+n}z^j
\]
and 
\[
s_\beta^n(z) = \sum_{k \ge 0} j(j-1)\ldots(j-n+1) \bb^{j-n}z^j.
\]
Further, in $H^2_w$, we have
\[
k_\beta^{(n)}(z) = \sum_{j\ge 0} (j+1)(j+2)\ldots(j+n)\frac{\bb^{j+n}z^j}{w_j}
\]
and 
\[
k_\beta^n(z) = \sum_{j \ge 0} j(j-1)\ldots(j-n+1) \frac{\bb^{j-n}z^j}{w_j}.
\]
Let us now relate the Maclaurin series coefficients of $s_\beta^{(n)}$ and $s_\beta^n$.

\begin{lemma}\label{span}
Let $F_0(j) = P_0(j) = 1$. For each $N \in \mathbb{Z}^+$, define $F_N(j):= \prod_{n=1}^N (j+n)$ and $P_N(j) := \prod_{n=0}^{N-1} (j-n)$. Then $F_N \in \text{span}\{P_0, \ldots, P_N\}$ for all $N \in \mathbb{Z}^+ \cup \{0\}$. 
\end{lemma}
\begin{proof}
We will proceed by induction. Let $N=1$ and observe $F_1(j) = j+1 = P_1(j) + P_0(j)$, so the base case holds. Now suppose $F_N \in \text{span}\{P_0, \ldots, P_N\}$ and note that $F_{N+1}(j) = (j+N+1) F_N(j)$. 
By the induction hypothesis, we can find constants $c_i$ such that 
\begin{align*}
F_{N+1}(j) &= (j+N+1) F_N(j)\\
&= j\sum_{i=0}^N c_i P_i(j) + (N+1)\sum_{i=0}^N c_i P_i(j). 
\end{align*}
Observe, for any $n \ge 0$, that $jP_n(j) = (j-n)P_n(j) + nP_n(j) = P_{n+1}(j) + nP_n(j)$. Hence,  $jP_n \in \text{span}\{P_0, \ldots, P_{n+1} \}$ and also $j\sum_{i=0}^N c_i P_i \in \text{span}\{ P_0, \ldots, P_{N+1} \}$. Thus, $F_{N+1} \in \text{span}\{P_0, \ldots, P_{N+1}\}$.
\end{proof}

\begin{remark}
The purpose of this lemma, as an immediate corollary, is  that 
\[
s_\beta^{(n)} \in \text{span}\{ s_\beta, s_\beta^1, \ldots, s_\beta^n\}.
\]
\end{remark}
We may now state and prove our final theorem, \editB{and mention again that inner functions arising as linear combinations of reproducing kernels goes back to work of Shapiro and Shields \cite{MR145082}.}

\begin{theorem}\label{projun}
Let $f \in H^2_w$ be a monic polynomial with $f(0) \neq 0$. Suppose $f$ has zeros $\beta_1, \ldots, \beta_r$ with multiplicities $m_1, \ldots, m_r$, respectively. Let $Z_j := \{ \beta_i \in Z(f) \cap \Omega_j : m_i > j \}$ be the set of zeros of $f$ that are reproducible of order $j$ and have multiplicity greater than $j$. Let $I_j := \{ i \in \{ 1, \ldots, r\} : \beta_i \in Z_j \}$ be the set of indices appearing in $Z_j$. Let $R:= \max(\{ j: Z_j \neq \emptyset\})$ be the largest value of $j$ such that $Z_j$ is non-empty.
Let $\varphi$ be the orthogonal projection of $1$ onto $[f]$. Then  
\[
\varphi(z) = 1 + \sum_{j=0}^{R} \  \sum_{i \in I_j} C_{i,j} k_{\beta_i}^{j}(z),
\]
where $k_\beta^i$ denotes the reproducing kernel for $i$-th derivatives in $H^2_w$ at $\beta$ and $C_{i,j}$ are constants determined by $\langle \varphi, k_{\beta_i}^j \rangle_w = 0$ for each $i \in I_j$ and $0 \le j \le R$. 
\end{theorem}

\begin{proof}
Put $f(z) = z^d + a_{d-1}z^{d-1} + \dots + a_0$ and denote the \editB{Maclaurin} coefficients of $\varphi$ as \editB{$\varphi_n = \langle \varphi, z^n \rangle_w / \| z^n \|_w^2$}. 
Since $\varphi \in \ksf$, we have $\langle \varphi, z^{n+d} + a_{d-1}z^{n + d-1} + \dots + a_0 z^n \rangle_w = 0$ for all $n \ge 1$. This gives the recurrence relation 
\[
w_{n+d} \varphi_{n+d} = \sum_{j = 0}^{d-1} -w_{n+j} \  \overline{a}_j  \ \varphi_{n+ j}. 
\]
Now let us use $\Phi_n : = w_n \varphi_n$ to obtain the constant coefficient recurrence relation
\[
\Phi_{n+d} = \sum_{j = 0}^{d-1}   -\overline{a}_j  \Phi_{n + j}. 
\]
We will now find the generating function $\Phi (z)$ (viewed as a \textit{formal} power series) by summing over all $n$ (see, e.g., \cite[Chapter~2]{erickson2013introduction} for more on solving recurrence relations and generating functions): 
\begin{align*}
\Phi(z) &= p(z) + \sum_{n \ge 0} -\overline{a_{d-1}}\Phi_{n-1}z^n + \dots + \sum_{n \ge 0} -\overline{a_0}\Phi_{n-d}z^n \\
& = p(z) - \overline{a_{n-1}}z\Phi(z) - \dots - \overline{a_0}z^d \Phi(z).
\end{align*}
where $p$ is a polynomial of degree $d$ given by the initial conditions of the relation.
Solving for $\Phi(z)$ gives 
\begin{align*}
\Phi(z) &= \frac{p(z)}{1 + \overline{a_{d-1}}z + \dots \overline{a_0}z^d} \\
&= \frac{p(z)}{z^d \overline{f(1/\overline{z})}}\\
&= \frac{p(z)}{\prod_{i = 1}^r (1 - \overline{\beta_i}z)^{m_i}}.
\end{align*}
After doing long division (because $\deg{p} = d$) and using partial fractions, with constants $C$ and $ c_{i,j}$, we may put
\begin{align*}
\Phi(z) &= C + \sum_{i = 1}^{r} \sum_{j = 1}^{m_i} \frac{c_{i,j}}{(1 - \overline{\beta_i}z)^j}\\
&= C + \sum_{i = 1}^{r} \sum_{j = 1}^{m_i} \frac{c_{i,j}}{\overline{\beta_i}^{j - 1}(j-1)!} \ \editB{D^{j-1}}\left(\frac{1}{1 - \overline{\beta_i}z}\right),
\end{align*}
\editB{where $D$ is the derivative operator with respect to $z$.}
Putting $\tilde{C}_{i,j} = \frac{c_{i,j}}{\overline{\beta_i}^{j - 1}(j-1)!}$ and substituting in with terms of $s_\beta^{(j)}$, we get
\[
\Phi(z) = C +  \sum_{i = 1}^{r} \sum_{j = 1}^{m_i} \tilde{C}_{i,j} s^{(j-1)}_{\beta_i}(z). 
\]
By Lemma \ref{span}, we can find constants $C_{i,j}$ such that
\[
\Phi(z) = C +  \sum_{i = 1}^{r} \sum_{j = 1}^{m_i} C_{i,j} s^{j-1}_{\beta_i}(z). 
\]
The upshot of going through the trouble of writing $\Phi$ in this way is that when substituting back in with $\varphi_n$, each term of the form $s^{j-1}_{\beta_i}$ becomes $k^{j-1}_{\beta_i}$. Doing so, we find the formal power series
\[
\tilde{\varphi}(z) = C +  \sum_{i = 1}^{r} \sum_{j = 1}^{m_i} C_{i,j} k^{j-1}_{\beta_i}(z).
\]
In order to find $\varphi$, we must determine which terms above converge in $H^2_w$. This is precisely when $\beta_i \in Z_j$, for appropriate $i,j$. Namely,
\[
\varphi(z) = C + \sum_{j=0}^{R} \  \sum_{i \in I_j} C_{i,j} k_{\beta_i}^{j}(z).
\]
Lastly, the claim about the constants follows by noting that any function in $[f]$ must vanish, with proper multiplicity, at the reproducible zeros of $f$.
If we let $F:=\sum_{j=0}^{R} \  \sum_{i \in I_j} C_{i,j} k_{\beta_i}^{j}$ and note that $\Pi_{[f]}F = 0$, then $\varphi = \Pi_{[f]}\varphi = \Pi_{[f]}C + \Pi_{[f]}F =  C\varphi$, so $C = 1$. As $\varphi \in [f]$, the other constants $C_{i,j}$ can also be determined by using the fact that $\varphi^{(j)}(\beta) = \langle \varphi, k_\beta^j \rangle_w = 0$ for $\beta \in Z_j$.
\end{proof}

\begin{remark}
The above theorem shows something stronger than what is stated; we have actually shown that $\ksf = \text{span}\{1, k_\beta^j : \beta \in Z_j\}$. Example \ref{simple} below gives an explicit linear system whose solution gives the constants appearing in $\varphi$ when $f$ has simple zeros.  An immediate corollary of the above theorem is that if $f, q \in \mathcal{P} \subset H^2_w$ with $Z(q) \cap \left(\cup_{m\ge 0} \Omega_m \right) = \emptyset$, then $\Pi_{[f]}(1) = \Pi_{[qf]}(1)$. This also tells us that the optimal approximants to $1/f$ and $1/(qf)$ form an equivalence class with respect to the limit of their approximants. That is, the equivalence $f \sim h$ if and only if $\Pi_{[f]}(1) = \Pi_{[h]}(1)$. We will call this the Roman equivalence relation; the approximants of two different functions in the same equivalence class travel along different roads, but end up in the same place. Another observation worth noting is that \editB{$\text{dist}^2(1, [f]) = 1- \varphi(0) = \sum_i C_{i,0}$}. This is due to the fact that $k_\beta(0) = 1$ and $k_\beta^j(0) = 0$ for all $j \ge 1$. 
\end{remark}

\section{Examples, Further Questions, and Discussion}

\editB{An immediate corollary of Theorem \ref{projun} is that a polynomial is cyclic in $\h$ if and only if it has no reproducible zeros. 
A natural desire would be to extend Theorem \ref{projun} to any function, not just a polynomial. 
Such an extension could possibly provide new information helpful for understanding cyclicity. 

We also mention again the hypotheses of Proposition \ref{cyclic-inner}; it is not known if the existence assumption is needed. Namely, one may ask, what additional hypotheses, if any, are required of $\h$ so that $k_0$ is cyclic? }

\editA{We conclude with a couple of examples. It is well known that for a function $f \in H^2$, $\Pi_{[f]}(1)$ is precisely the inner part of $f$. The first of these examples communicates this by applying Theorem \ref{projun}.}

\begin{example}
Let us consider  $f(z) = \prod_{i=1}^d (z - \beta_i) \in H^2$ with $f(0) \neq 0$. Let $\Omega = Z(f) \cap \D$. Since $\Omega_m(H^2) = \D$ for all $m \ge 0$, we know from Theorem \ref{projun} that $\varphi:= \Pi_{[f]}1$ is given by 
\[
\varphi(z) = \frac{p(z)}{\prod_{\beta \in \Omega}(1 - \overline{\beta}z)}.
\]
We also know that $p$ must vanish at each $\beta \in \Omega$ so we get, for some constant $C$, 
\[
\varphi(z) = C \prod_{\beta \in \Omega} \frac{(z - \beta)}{(1 - \overline{\beta}z)}.
\]
The fact that $\langle \varphi , \varphi \rangle  = \varphi(0)$ implies 
\[ 
|C|^2 \ \left\| \prod_{\beta \in \Omega} \frac{(z - \beta)}{(1 - \overline{\beta}z)} \right\|^2_{H^2} = C\ \prod_{\beta \in \Omega} (-\overline{\beta}).
\]
In turn, we have $C = \prod_{\beta \in \Omega} (-\overline{\beta})$.
This gives $\varphi$ as a multiple of a familiar Blaschke product (an $H^2$-inner function):
\[
\varphi(z) =  \prod_{\beta \in \Omega}(-\overline{\beta}) \frac{(z - \beta)}{(1 - \overline{\beta}z)}.
\]
\editB{This also tells us that $\operatorname{dist}^2(1, [f]) = 1 - \varphi(0) = 1 - \prod_{ \beta \in \Omega}|\beta|^2$. Further, when $\Omega$ is empty, we have that $\varphi  \equiv 1$, and $f$ is cyclic. }
\end{example}

\editA{Our last example is another application of Theorem \ref{projun}, which points out a computational improvement; finding an optimal approximant requires solving a linear system, while, in the polynomial case, invoking Theorem \ref{projun} allows us to compute the \textit{limit} of optimal approximants by solving a linear system. }

\begin{example}\label{simple}
Suppose $f \in H^2_w$ is a monic polynomial \editB{with} simple zeros and $f(0) \neq 0$. Let $\{ \beta_i \}_{1}^{d} = Z(f) \cap \Omega_0(H^2_w)$. Theorem \ref{projun} says the orthogonal projection of $1$ onto $[f]$ is given by
\[
\varphi(z) = 1 + \sum_{i = 1}^d C_i k_{\beta_i}(z)
\]
for some constants $C_i$. 
Since $\varphi$ vanishes at each $\beta_i$, we get, for $1 \le j \le d$, 
\[
 0 = \varphi(\beta_j) 
 = 1 + \sum_{i = 1}^d C_i k_{\beta_i}(\beta_j) 
 = 1 + \sum_{i = 1}^d C_i \langle k_{\beta_i}, k_{\beta_j} \rangle_w.
\]
Moving the independent term 1 to the left-hand side in each equation expressed above gives the linear system
\[
\left( \langle k_{\beta_i}, k_{\beta_j} \rangle_w \right)_{1 \le i,j \le d} \left( C_1, \ldots, C_d \right)^T = \left( -1, \ldots, -1\right)^T.
\]
\end{example}

\subsection*{Acknowledgment}
Many thanks to John McCarthy, Catherine B{\'e}n{\'e}teau, and Daniel Seco for helpful discussion. Additional thanks to the referees, whose comments greatly improved the content and presentation of this work.

\bibliographystyle{abbrv}
\bibliography{CFelder_OPA_VF}

\end{document}